\documentclass[11pt]{article}
\usepackage[cp1251]{inputenc}
\usepackage[english]{babel}
\usepackage{amsmath}
\usepackage{amsfonts}
\usepackage{mathrsfs}
\usepackage{amssymb}
\usepackage{graphicx}
\usepackage{multirow}
\usepackage{hanging}
\usepackage{relsize}
\usepackage[flushmargin, symbol*]{footmisc}
\usepackage[usenames,dvipsnames]{color}
\newtheorem{theorem}{\textsf{Theorem}}
\newtheorem{lemma}{\textsf{Lemma}}
\newtheorem{proposition}{\textsf{Proposition}}
\newenvironment{proof}[1][\textsf{Proof. }]{\textbf{#1}}{$\square$}

\newcommand{\fn}[1]{\footnote{\hangpara{3em}{1} #1}}
\begin{document}

\title{
\date{}
{
\large \textsf{\textbf{Volume of a doubly truncated hyperbolic tetrahedron}}
}
}
\author{\small Alexander Kolpakov\fn{supported by the Schweizerischer Nationalfonds SNF no.~200021-131967/1} \hspace*{42pt}
\small Jun Murakami\fn{supported by Grant-in-Aid for Scientific Research no.~22540236
}}
\maketitle

\begin{abstract}\noindent

The present paper regards the volume function of a doubly truncated hyperbolic tetrahedron. Starting from the previous results of J.~Murakami, U.~Yano and A.~Ushijima, we have developed a unified approach to express the volume in different geometric cases via dilogarithm functions and to treat properly the many analytic strata of the latter. Finally, several numeric examples are given.

\medskip
{\textsf{\textbf{Key words}}: hyperbolic tetrahedron, Gram matrix, volume formula.}
\end{abstract}

\parindent=0pt

\section{Introduction}

The real vector space $\mathbb{R}^{1,n}$ of dimension $n+1$ with the Lorentzian inner product $\langle x, y\rangle = -x_0 y_0 + x_1y_1 + \dots + x_n y_n$, where $x = (x_0, x_1, \dots, x_n)$ and $y = (y_0, y_1, \dots, y_n)$, is called an $(n+1)$-dimensional Lorentzian space $\mathbb{E}^{1,n}$.

Consider the two-fold hyperboloid $\mathscr{H} = \{x \in \mathbb{E}^{1,n} | \langle x, x \rangle = -1\}$ and its upper sheet $\mathscr{H}^{+} = \{x \in \mathbb{E}^{1,n} | \langle x, x \rangle = -1, x_0 > 0\}$. The restriction of the quadratic form induced by the Lorentzian inner product $\langle \circ, \circ \rangle$ to the tangent space to $\mathscr{H}^{+}$ is positive definite, and so it gives a Riemannian metric on $\mathscr{H}^{+}$. The space $\mathscr{H}^{+}$ equipped with this metric is called \textit{the hyperboloid model} of the $n$-dimensional hyperbolic space and denoted by $\mathbb{H}^n$. The hyperbolic distance $d(x,y)$ between two points $x$ and $y$ with respect to this metric is given by the formula $\cosh d = -\langle x,y \rangle$.

Consider the cone $\mathscr{K} = \{x \in \mathbb{E}^{1,n} | \langle x, x \rangle = 0\}$ and its upper half $\mathscr{K}^{+} = \{x \in \mathbb{E}^{1,n} | \langle x, x \rangle = 0, x_0 > 0\}$. A ray in $\mathscr{K}^{+}$ emanating from the origin corresponds to a point on the ideal boundary of $\mathbb{H}^n$. The set of such rays forms a sphere at infinity $\mathbb{S}^{n-1}_{\infty} \cong \partial \mathbb{H}^n$. Thus, each ray in $\mathscr{K}^{+}$ becomes an ideal point of $\overline{\mathbb{H}^n} = \mathbb{H}^n \cup \partial \mathbb{H}^n$.

Let $p$ denote the radial projection of $\mathbb{E}^{1,n} \setminus \{x \in \mathbb{E}^{1,n} | x_0 = 0\}$ onto the affine hyperplane $\mathbb{P}^n_1 = \{x \in \mathbb{E}^{1,n} | x_0 = 1\}$ along a ray emanating from the origin~$\mathbf{o}$. The projection $p$ is a diffeomorphism of $\mathbb{H}^n$ onto the open $n$-dimensional unit ball $\mathbb{B}^n$ in $\mathbb{P}^n_1$ centred at $(1, 0, 0, \dots, 0)$ which defines a projective model of $\mathbb{H}^n$. The affine hyperplane $\mathbb{P}^n_1$ includes not only $\mathbb{B}^n$ and its set-theoretic boundary $\partial \mathbb{B}^n$ in $\mathbb{P}^n_1$, which is canonically identified with $\mathbb{S}^{n-1}_{\infty}$, but also the exterior of \textit{the compactified projective model} $\overline{\mathbb{B}^n} = \mathbb{B}^n \cup \partial \mathbb{B}^n \cong \mathbb{H}^n \cup \mathbb{S}^{n-1}_{\infty}$. Let $\mathrm{Ext}\, \mathbb{B}^n$ denote the exterior of $\overline{\mathbb{B}^n}$ in $\mathbb{P}^n$. Thus $p$ could be naturally extended to a map from $\mathbb{E}^{1,n}\setminus \mathbf{o}$ onto an $n$-dimensional real projective space $\mathbb{P}^n = \mathbb{P}^n_1 \cup \mathbb{P}^n_{\infty}$, where $\mathbb{P}^n_{\infty}$ is the set of straight lines in the affine hyperplane $\{x \in \mathbb{E}^{1,n} | x_0 = 0\}$ passing through the origin. 

Consider the one-fold hyperboloid $\mathscr{H}_{\star} = \{x \in \mathbb{E}^{1,n} | \langle x, x \rangle = 1\}$. Given some point $u$ in $\mathscr{H}_{\star}$ define in $\mathbb{E}^{1,n}$ the half-space $\mathrm{R}_u = \{x \in \mathbb{E}^{1,n} | \langle x, u \rangle \leq 0\}$ and the hyperplane $\mathrm{P}_u = \{x \in \mathbb{E}^{1,n} | \langle x, u \rangle = 0\} = \partial \mathrm{R}_u$. Denote by $\Gamma_u$ (respectively $\Pi_u$) the intersection of $\mathrm{R}_u$ (respectively $\mathrm{P}_u$) with $\mathbb{B}^n$. Then $\Pi_u$ is a geodesic hyperplane in $\mathbb{H}^n$, and the correspondence between the points in $\mathcal{H}_{\star}$ and the half-space $\Gamma_u$ in $\mathbb{H}^n$ is bijective. Call the vector $u$ normal to the hyperplane $\mathrm{P}_u$ (or $\Pi_u$).

Let $v$ be a point in $\mathrm{Ext}\, \mathbb{B}^n$. Then $p^{-1}(v)\cap \mathscr{H}_{\star}$ consists of two points. Let $\tilde{v}$ denote one of them, so we may define \textit{the polar hyperplane} $\Pi_{\tilde{v}}$ to $v$, independent on the choice of $\tilde{v} \in p^{-1}(v)\cap \mathscr{H}_{\star}$.

Now we descend to dimension $n=3$. Let $T$ be a tetrahedron in $\mathbb{P}^3_1$, that is a convex hull of four points $\{v_i\}^{4}_{i=1} \subset \mathbb{P}^3_1$. We say a vertex $v \in \{v_i\}^{4}_{i=1}$ to be \textit{proper} if $v \in \mathbb{B}^3$, to be \textit{ideal} if $v \in \partial \mathbb{B}^3$ and to be \textit{ultra-ideal} if $v\in \mathrm{Ext}\, \mathbb{B}^3$. 

Let $v$ be an ultra-ideal vertex of $T$. We call \textit{a truncation of} $v$ the operation of
removing the pyramid with apex $v$ and base $\Pi_{\tilde{v}}\cap T$. A \textit{generalised} hyperbolic tetrahedron $T$ is a polyhedron, possibly non-compact, of finite volume in the hyperbolic space obtained from a certain tetrahedron by polar truncation of its ultra-ideal vertices. In case when only two vertices are truncated, we call such a generalised tetrahedron \textit{doubly truncated}. 

Depending on the dihedral angles, the polar hyperplanes may or may not intersect. In Fig.~\ref{tetrahedron0} a tetrahedron $T$ with two truncated vertices is depicted. The corresponding polar planes do not intersect. We shall call this kind of generalised tetrahedron \textit{mildly truncated}. In Fig.~\ref{tetrahedron} the case when the polar planes intersect is shown. Here the tetrahedron $T$ is truncated down to a prism, hence we shall call it \textit{prism truncated}.

In this paper we study the volume function of a doubly truncated hyperbolic tetrahedron. The question arises first in the paper by R.~Kellerhals \cite{Kellerhals}, where a doubly truncated hyperbolic orthoscheme is considered. The whole evolution of an orthoscheme, starting from a mildly truncated one down to a Lambert cube has been investigated. 

The case of a general hyperbolic tetrahedron was considered in numerous papers \cite{ChoKim, DM, MY}. The case of a mildly truncated tetrahedron is due to J.~Murakami and A.~Ushijima \cite{MU, Ushijima}. The paper \cite{Ushijima} is the starting point where the question about intense truncations of a hyperbolic tetrahedron was posed. Thus, the case of a prism truncated tetrahedron remains unattended. As we shall see later, it brings some essential difficulties. First, the structure of the volume formula should change, as first observed in \cite{Kellerhals}. Second, the branching properties of the volume function come into sight. This phenomenon was first observed for tetrahedra in the spherical space and is usually related to the use of the dilogarithm function or its analogues, see \cite{KMP, MU, Murakami}.

For the rest of the paper, a mildly (doubly) truncated tetrahedron is given in Fig.~\ref{tetrahedron0}. Its dihedral angles are $\theta_k$ and its corresponding edge lengths are $\ell_k$, $k=\overline{1,6}$\footnote{for the given integers $n\geq 1$, $m\geq 0$, the notation $k=\overline{n,n+m}$ means that $k\in \{n, n+1, \dots, n+m\}$.}. 

\begin{figure}[ht]
\begin{center}
\includegraphics* [totalheight=6cm]{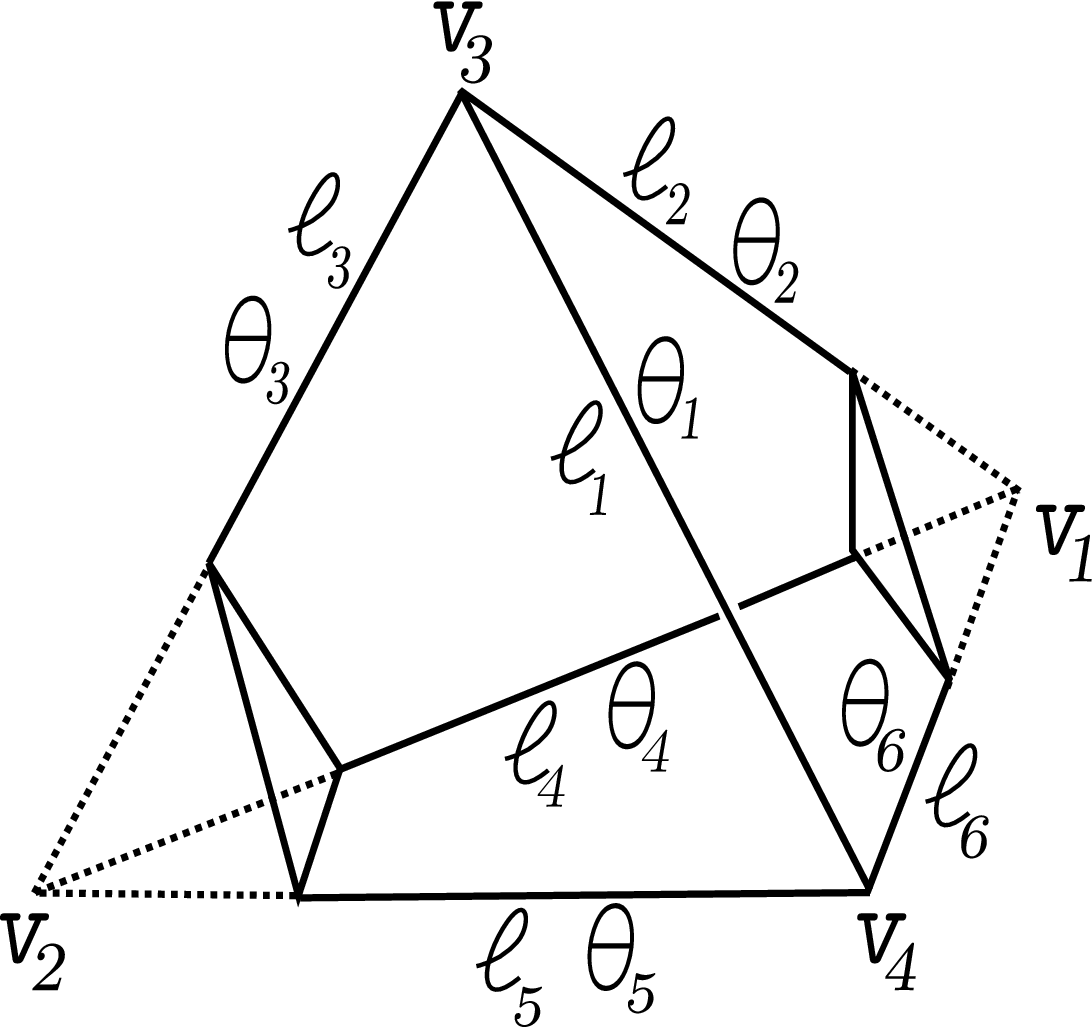}
\end{center}
\caption{Mildly (doubly) truncated tetrahedron} \label{tetrahedron0}
\end{figure}

A prism truncated tetrahedron is given in Fig.~\ref{tetrahedron}. The dashed edge connects the ultra-ideal vertices and corresponds to the edge $\ell_4$ in the previous case. The dihedral angles remain the same, except that the altitude $\ell$ of the prism replaces the dihedral angle $\theta_4$. The altitude carries the dihedral angle $\mu$ and is orthogonal to the bases because of the truncation. Right dihedral angles in Fig.~\ref{tetrahedron0} -- \ref{tetrahedron} are not indicated by symbols. The other dihedral angles and the corresponding edge length are called \textit{essential}. 

\begin{figure}[ht]
\begin{center}
\includegraphics* [totalheight=5.5cm]{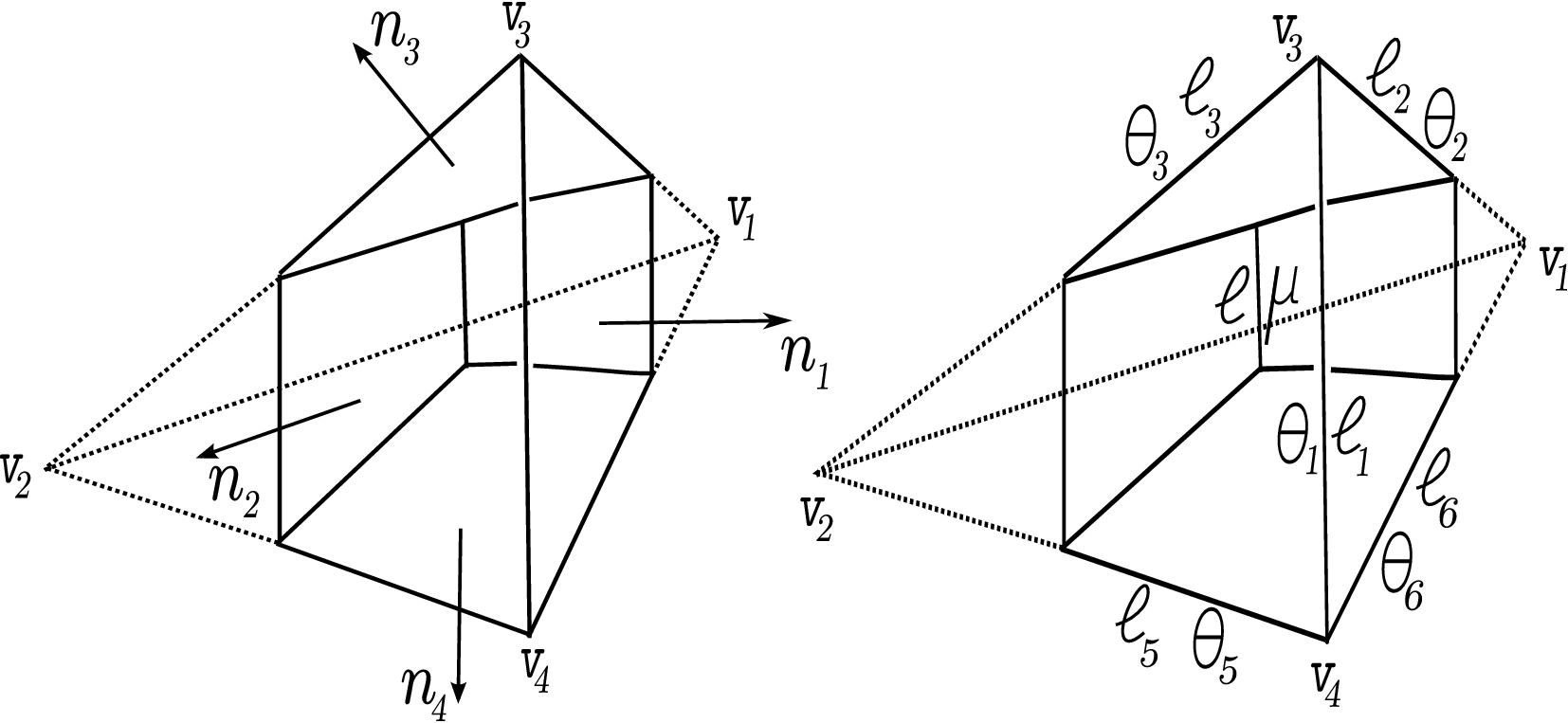}
\end{center}
\caption{Prism (intensely) truncated tetrahedron $T$ and its geometric parameters} \label{tetrahedron}
\end{figure}

\medskip
\textbf{Acknowledgement.} The first named author is grateful to Waseda University, Tokyo for hospitality during his stay in December, 2011. 
Part of this work was performed at the Institut Henri Poincar\'{e}, Paris in February-March, 2012.

\section{Preliminaries}

The following propositions reveal a relationship between the Lorentzian inner product of two vectors and the mutual position of the respective hyperplanes.

\begin{proposition}\label{propositionHyperplanes1}
Let $u$ and $v$ be two non-collinear points in $\mathscr{H}_{\star}$. Then the following holds:
\begin{itemize}
\item[i)] The hyperplanes $\Pi_{u}$ and $\Pi_{v}$ intersect if and only if $|\langle u,v \rangle| < 1$. The dihedral angle $\theta$ between them measured in $\Gamma_u \cap \Gamma_v$ is given by the formula $\cos \theta = -\langle u,v \rangle$.
\item[ii)] The hyperplanes $\Pi_{u}$ and $\Pi_{v}$ do not intersect in $\overline{\mathbb{B}^n}$ if and only if $|\langle u,v \rangle| > 1$. They intersect in $\mathrm{Ext}\,\mathbb{B}^n$ and admit a common perpendicular inside $\mathbb{B}^n$ of length $d$ given by the formula $\cosh d = \pm \langle u,v \rangle$\footnote{we chose the minus sign if $\Gamma_u\cap \Gamma_v = \emptyset$, otherwise we choose the plus sign.}.
\item[iii)] The hyperplanes $\Pi_{u}$ and $\Pi_{v}$ intersect on the ideal boundary $\partial \mathbb{B}^n$ only, if and only if $|\langle u,v \rangle| = 1$. 
\end{itemize}
In case ii) we say the hyperplanes $\Pi_{u}$ and $\Pi_{v}$ to be \textit{ultra-parallel} and in case iii) to be \textit{parallel}.
\end{proposition}

\begin{proposition}\label{propositionHyperplanes2}
Let $u$ be a point in $\mathbb{B}^n$ and let $\Pi_v$ be a geodesic hyperplane whose normal vector $v \in \mathscr{H}_{\star}$ is such that $\langle u,v \rangle < 0$. Then the length $d$ of the perpendicular dropped from the point $u$ onto the hyperplane $\Pi_v$ is given by the formula $\sinh d = -\langle u,v \rangle$.
\end{proposition}

Let $T$ be a generalised tetrahedron in $\mathbb{H}^3$ with outward Lorentzian normals $n_i$, $i=\overline{1,4}$, to its faces and vertex vectors $v_i$, $i=\overline{1,4}$, as depicted in Fig.~\ref{tetrahedron}. Let $G$ denote the Gram matrix for the normals $G = \langle n_i, n_j \rangle^{4}_{i,j=1}$.

The conditions under which $G$ describes a generalised mildly truncated tetrahedron are given by \cite{MP} and \cite{Ushijima}. For a prism truncated tetrahedron, its existence and geometry are determined by the vectors $\tilde{n}_i$, $i=\overline{1,6}$, where $\tilde{n}_i = n_i$ for $i=\overline{1,4}$ and $\tilde{n}_5 = v_1$, $\tilde{n}_6 = v_2$. Hence the above matrix $G$ does not have necessarily a Lorentzian signature. However, it will suffice for our purpose.

In case $T$ is prism truncated, as Fig.~\ref{tetrahedron} shows, we obtain
\begin{equation}\label{eqGramMatrix}
G = \left( \begin{array}{cccc}
1 & -\cos \theta_1 & -\cos \theta_2 & -\cos \theta_6\\
-\cos \theta_1 & 1 & -\cos \theta_3 & -\cos \theta_5\\
-\cos \theta_2 & -\cos \theta_3 & 1 & -\cosh \ell\\
-\cos \theta_6 & -\cos \theta_5 & -\cosh \ell & 1
\end{array} \right),
\end{equation}
in accordance with Proposition~\ref{propositionHyperplanes1}.

We will \textit{define} the edge length matrix $G^{\star}$ of $T$ by
\begin{equation}\label{eqLengthMatrx}
G^{\star} = \left( \begin{array}{cccc}
-1 & -\cos \mu & i \sinh \ell_5    & i \sinh \ell_3\\
-\cos \mu & -1 & i \sinh \ell_6    & i \sinh \ell_2\\
-i \sinh \ell_5 & -i \sinh \ell_6  & -1 & - \cosh \ell_1\\
-i \sinh \ell_3 & -i \sinh \ell_2  & - \cosh \ell_1 & -1
\end{array} \right)
\end{equation}
in order to obtain a Hermitian analogue of the usual edge length  matrix for a mildly truncated tetrahedron (for the latter, see \cite{Ushijima}). 

By \cite{MP} or \cite{Ushijima}, we have that
\begin{equation}\label{eqGramMinors} 
- g^{\star}_{ij} = \frac{c_{ij}}{\sqrt{c_{ii}} \sqrt{c_{jj}}},
\end{equation}
where $c_{ij}$ are the corresponding $(i,j)$ cofactors of the matrix $G$. The complex conjugation $g^{\star}_{ij} = \overline{g^{\star}_{ji}}$, $i,j=\overline{1,4}$, corresponds to a choice of the analytic strata for the square root function $\sqrt{\circ}$. 

\begin{proposition}[\cite{MP, Ushijima}]\label{propositionVertices}
The vertex $v_i$, $i=\overline{1,4}$, of $T$ is proper, ideal, or ultra-ideal provided that $c_{ii} > 0$, $c_{ii} = 0$, or $c_{ii} < 0$, respectively.
\end{proposition}

Hence, Propositions \ref{propositionHyperplanes1}-\ref{propositionHyperplanes2} imply that the matrices $G$ and $G^{\star}$ agree concerning the relationship of the geometric parameters of $T$. The matrix $G^{\star}$ has complex entries since the minors $c_{ii}$, $i=\overline{1,4}$, in formula (\ref{eqGramMinors}) may have different signs by Proposition~\ref{propositionVertices}.

To perform our computations later on in a more efficient way, we shall introduce the parameters $a_k$, $k=\overline{1,6}$, associated with the edges of the tetrahedron $T$. If $T$ is a prism truncated tetrahedron, then we set $a_k := e^{i \theta_k}$, $k\in\{1,2,3,5,6\}$, $a_4 := e^{\ell}$ and then
\begin{equation}\label{eqGramMatrixAlg}
G = \left( \begin{array}{cccc}
1 & -\frac{a_1 + 1/a_1}{2} & -\frac{a_2 + 1/a_2}{2} & -\frac{a_6 + 1/a_6}{2}\\
-\frac{a_1 + 1/a_1}{2} & 1 & -\frac{a_3 + 1/a_3}{2} & -\frac{a_5 + 1/a_5}{2}\\
-\frac{a_2 + 1/a_2}{2} & -\frac{a_3 + 1/a_3}{2} & 1 & -\frac{a_4 + 1/a_4}{2}\\
-\frac{a_6 + 1/a_6}{2} & -\frac{a_5 + 1/a_5}{2} & -\frac{a_4 + 1/a_4}{2} & 1
\end{array} \right).
\end{equation}

The meaning of the parameters becomes clear if one observes the picture of a mildly truncated tetrahedron $T$ (see Fig.~\ref{tetrahedron}). In case two vertices $v_1$ and $v_2$ of $T$ become ultra-ideal and the corresponding polar hyperplanes intersect (see \cite[Sections 2-3]{Ushijima} for more basic details), the edge $v_1v_2$ becomes dual in a sense to the altitude $\ell$ of the resulting prism. Since in case of a mildly truncated tetrahedron we have $a_k = e^{i \theta_k}$ according to \cite{Ushijima}, the altitude $\ell$ for now corresponds still to the parameter $a_4$, in order to keep consistent notation.
 
\section{Volume formula}

Let $\mathscr{U} = \mathscr{U}(a_1,a_2,a_3,a_4,a_5,a_6,z)$ denote the function
\begin{eqnarray}\label{eqUDefinition}
\lefteqn{\mathscr{U} = \mathrm{Li}_2(z) + \mathrm{Li}_2(a_1a_2a_4a_5z) + \mathrm{Li}_2(a_1a_3a_4a_6z) + \mathrm{Li}_2(a_2a_3a_5a_6z)}\\
\nonumber && - \mathrm{Li}_2(-a_1a_2a_3z) - \mathrm{Li}_2(-a_1a_5a_6z) - \mathrm{Li}_2(-a_2a_4a_6z) - \mathrm{Li}_2(-a_3a_4a_5z)
\end{eqnarray}
depending on seven complex variables $a_k$, $k=\overline{1,6}$ and $z$, where $\mathrm{Li}_2(\circ)$ is the dilogarithm function.
Let $z_{-}$ and $z_{+}$ be two solutions to the equation $e^{z \frac{\partial \mathscr{U}}{\partial z}} = 1$ in the variable $z$. According to \cite{MY}, these are
\begin{equation}\label{eqZpm}
z_{-} = \frac{-q_1-\sqrt{q^2_1-4q_0q_2}}{2q_2} \,\,\,\mbox{ and }\,\,\, z_{+} = \frac{-q_1+\sqrt{q^2_1-4q_0q_2}}{2q_2},
\end{equation}
where
\begin{equation*}
q_0 = 1 + a_1a_2a_3 + a_1a_5a_6 + a_2a_4a_6 + a_3a_4a_5
+ a_1a_2a_4a_5 + a_1a_3a_4a_6 + a_2a_3a_5a_6,
\end{equation*}
\begin{eqnarray}\label{eqQ1Q2Q3Definition}
\nonumber q_1 = -a_1 a_2 a_3 a_4 a_5 a_6 \bigg{(}\bigg{(}a_1-\frac{1}{a_1}\bigg{)}\bigg{(}a_4-\frac{1}{a_4}\bigg{)} + \bigg{(}a_2-\frac{1}{a_2}\bigg{)}\bigg{(}a_5-\frac{1}{a_5}\bigg{)}\\
+\bigg{(}a_3-\frac{1}{a_3}\bigg{)}\bigg{(}a_6-\frac{1}{a_6}\bigg{)}\bigg{)},
\end{eqnarray}
\begin{eqnarray*}
q_2 = a_1 a_2 a_3 a_4 a_5 a_6 (a_1 a_4 + a_2 a_5 + a_3 a_6 + a_1a_2a_6 + a_1a_3a_5 + a_2a_3a_4 + \\a_4a_5a_6
+ a_1a_2a_3a_4a_5a_6).
\end{eqnarray*}

Given a function $f(x,y,\dots,z)$, let $f(x,y,\dots,z)\mid^{z=z_{-}}_{z=z_{+}}$ denote the difference $f(x,y,\dots,z_{-}) - f(x,y,\dots,z_{+})$. Now we define the following 
function $\mathscr{V} = \mathscr{V}(a_1,a_2,a_3,a_4,a_5,a_6,z)$ by means of the equality
\begin{equation}\label{eqVDefinition}
\mathscr{V} = \frac{i}{4}\left( \mathscr{U}(a_1,a_2,a_3,a_4,a_5,a_6,z) - z\, \frac{\partial \mathscr{U}}{\partial z}\, \log z \right)\bigg{\vert}^{z=z_{-}}_{z=z_{+}}.
\end{equation}

Let $\mathscr{W} = \mathscr{W}(a_1,a_2,a_3,a_4,a_5,a_6,z)$ denote the function below, that will correct possible branching of $\mathscr{V}$ resulting from the use of (di-)logarithms:
\begin{equation}\label{eqWDefinition}
\mathscr{W} = \sum^6_{k=1} \left( a_k\, \frac{\partial \mathscr{V}}{\partial a_k} - \frac{i}{4}\, \log e^{-4 i\, a_k\, \frac{\partial \mathscr{V}}{\partial a_k}} \right) \log a_k.
\end{equation}

Given a generalised hyperbolic tetrahedron as in Fig.~\ref{tetrahedron}, truncated down to a quadrilateral prism with essential dihedral angles $\theta_k$, $k\in\{1,2,3,5,6\}$, altitude $\ell$ and dihedral angle $\mu$ along it, we set $a_k = e^{i \theta_k}$, $k\in\{1,2,3,5,6\}$, $a_4 = e^{\ell}$, as above. Then the following theorem holds.

\begin{theorem}\label{thmVolume}
Let $T$ be a generalised hyperbolic tetrahedron as given in Fig.~\ref{tetrahedron}. Its volume equals
\begin{equation*}
\mathrm{Vol}\,T = \Re \left(-\mathscr{V} + \mathscr{W} - \frac{\mu \ell}{2}\right).
\end{equation*}
\end{theorem}

\bigskip
\textbf{Note.} In the statement above, the altitude length is $\ell = \Re \log a_4$ and the corresponding dihedral angle equals
$\mu = -2\, \Re \left( a_4 \frac{\partial \mathscr{V}}{\partial a_4} \right)\, \mathrm{mod}\, \pi$.

\bigskip
\textbf{3.1} \textbf{Preceding lemmas.} Before giving a proof to Theorem~\ref{thmVolume}, we need several auxiliary statements concerning the branching of the volume function.

\begin{lemma}\label{lemmaWTerm}
The function $\mathscr{W}$ has a.e. vanishing derivatives $\frac{\partial \mathscr{W}}{\partial a_k}$, $k=\overline{1,6}$.
\end{lemma}
\begin{proof}
By computing the derivative of (\ref{eqWDefinition}) with respect to each $a_k$, $k=\overline{1,6}$, outside of its branching points and by making use of the identities 
$\frac{\mathrm{d}}{\mathrm{d}z}\log z = 1/z$, $e^{z} e^{w} = e^{z+w}$ for all $z,w \in \mathbb{C}$. Since the branching points of a finite amount of 
$\log(\circ)$ and $\mathrm{Li}_2(\circ)$ functions form a discrete set in $\mathbb{C}$, the lemma follows.
\end{proof}

\begin{lemma}\label{lemmaNoBranching}
The function $\Re\left(-\mathscr{V}+\mathscr{W}\right)$ does not branch with respect to the variables $a_k$, $k=\overline{1,6}$, and $z$.
\end{lemma}
\begin{proof}
Let us consider a possible branching of the function defined by formula (\ref{eqUDefinition}). Let $\mathscr{U}$ comprise only principal strata of the dilogarithm and let $\mathscr{U}^{\star}$ correspond to another ones. Then we have
\begin{eqnarray}\label{branchingU}
\nonumber \mathscr{U}\mid_{z=z_{\pm}} = \mathscr{U}^{\star}\mid_{z=z_{\pm}} + 2 \pi i k^{\pm}_0 \log(z_{\pm}) + 2 \pi i k^{\pm}_1 \log(a_1a_2a_4a_5z_{\pm})\\ 
\nonumber + 2 \pi i k^{\pm}_2 \log(a_1a_3a_4a_6z_{\pm}) + 2 \pi i k^{\pm}_3 \log(a_2a_3a_5a_6z_{\pm})\\
+ 2 \pi i k^{\pm}_4 \log(-a_1a_2a_3z_{\pm}) + 2 \pi i k^{\pm}_5 \log(-a_1a_5a_6z_{\pm})\\
\nonumber + 2 \pi i k^{\pm}_6 \log(-a_2a_4a_6z_{\pm}) + 2 \pi i k^{\pm}_7 \log(-a_3a_4a_5z_{\pm})\\ 
\nonumber + 4 \pi^2 k^{\pm}_8,
\end{eqnarray}
with some $k_j \in \mathbb{Z}$, $j=\overline{0,8}$.
From the above formula, it follows that
\begin{equation}\label{branchingdUdz}
z_{\pm} \frac{\partial \mathscr{U}}{\partial z}\bigg{|}_{z=z_{\pm}} \log z_{\pm} = z_{\pm} \frac{\partial \mathscr{U}^{\star}}{\partial z}\bigg{|}_{z=z_{\pm}} \log z_{\pm} + 2 \pi i \sum^7_{j=0} k^{\pm}_j \log z_{\pm}. 
\end{equation}
Then, according to formulas (\ref{eqVDefinition}), (\ref{branchingU}) -- (\ref{branchingdUdz}), the following expression holds for the corresponding analytic strata of the function $\mathscr{V}$:
\begin{equation}\label{branchingV}
\mathscr{V} = \mathscr{V}^{\star} - \frac{\pi}{2} \sum^{6}_{j=1} m_j \log a_j + \frac{i \pi^2}{2} m_7,
\end{equation}
where $m_{j} \in \mathbb{Z}$, $j=\overline{1,7}$ and we have used the formula $\log( u v ) = \log u + \log v + 2\pi i k$, $k \in \mathbb{Z}$. Hence, according to (\ref{branchingV}), we compute
\begin{equation*}
a_j \frac{\partial \mathscr{V}}{\partial a_j} = a_j \frac{\partial \mathscr{V}^{\star}}{\partial a_j} - \frac{\pi m_j}{2},
\end{equation*}
for each $j=\overline{1,6}$. The latter implies that
\begin{equation*}
a_j \frac{\partial \mathscr{V}}{\partial a_j} - \frac{i}{4} \log\,e^{-4 i\, a_j\, \frac{\partial \mathscr{V}}{\partial a_j}} = 
a_j \frac{\partial \mathscr{V}^{\star}}{\partial a_j} - \frac{i}{4} \log\,e^{-4 i\, a_j\, \frac{\partial \mathscr{V}^{\star}}{\partial a_j}}  - \frac{\pi m_j}{2},
\end{equation*}
for $j=\overline{1,6}$, since we choose the principal stratum of the logarithm function $\log(\circ)$. Thus, by formula (\ref{eqWDefinition}), 
\begin{eqnarray}
\nonumber
-\mathscr{V} + \mathscr{W} = - \mathscr{V}^{\star} + \mathscr{W}^{\star} + \frac{\pi}{2} \sum^{6}_{j=1} m_j \log a_j - \frac{i \pi^2}{2} m_7 - \frac{\pi}{2} \sum^{6}_{j=1} m_j \log a_j\\
\nonumber 
= - \mathscr{V}^{\star} + \mathscr{W}^{\star} - \frac{i \pi^2}{2} m_7,
\end{eqnarray}
with $m_j\in \mathbb{Z}$, $j=\overline{1,7}$. The proof is completed.
\end{proof}

\bigskip
\textbf{3.2} \textbf{Proof of Theorem~\ref{thmVolume}.} The scheme of our proof is the following: first we show that 
$\frac{\partial}{\partial \theta_k} \mathrm{Vol}\,T = - \frac{\ell_k}{2}$, $k\in\{1,2,3,5,6\}$, $\frac{\partial}{\partial \mu} \mathrm{Vol}\,T = - \frac{\ell}{2}$ 
and second we apply the generalised Schl\"{a}fli formula \cite[Equation~1]{MilnorSchlaefli} to show that the volume function and the one from Theorem~\ref{thmVolume} coincide up to a constant. 
Finally, the remaining constant is determined.

\medskip
Now, let us prove the three statements below:
\begin{itemize}
\item[(i)] $\frac{\partial}{\partial \theta_1} \mathrm{Vol}\,T = -\frac{\ell_1}{2}$,
\item[(ii)] $\frac{\partial}{\partial \theta_k} \mathrm{Vol}\,T = -\frac{\ell_k}{2}$ for $k\in\{2,3,5,6\}$,
\item[(iii)] $\frac{\partial}{\partial \mu} \mathrm{Vol}\,T = -\frac{\ell}{2}$.
\end{itemize}

\medskip
Note that in case (ii) it suffices to show $\frac{\partial}{\partial \theta_2} \mathrm{Vol}\,T = -\frac{\ell_2}{2}$. The statement for another $k\in\{3,5,6\}$ is completely analogous.

\smallskip
Let us show that the equality in case (i) holds. First, we compute 
\begin{equation*}
\frac{\partial}{\partial \theta_1} \left( \mathscr{U} - z\, \frac{\partial \mathscr{U}}{\partial z}\, \log z \right) = i\, a_1 \,\frac{\partial}{\partial a_1} \left( \mathscr{U} - z\, \frac{\partial \mathscr{U}}{\partial z}\, \log z \right) = 
\end{equation*}
\begin{equation*}
= i a_1 \left( \frac{\partial \mathscr{U}}{\partial a_1} - \log z \frac{\partial}{\partial a_1}\left( z \frac{\partial \mathscr{U}}{\partial z} \right) \right)
\end{equation*}

Upon the substitution $z := z_{\pm}$, we see that
\begin{equation*}
\frac{\partial}{\partial a_1}\left( z_{\pm} \frac{\partial \mathscr{U}}{\partial z}(a_1,a_2,a_3,a_4,a_5,a_6,z_{\pm}) \right) = 0,
\end{equation*}
by taking the respective derivative on both sides of the identity 
\begin{equation*}
e^{z_{\pm} \frac{\partial \mathscr{U}}{\partial z}(a_1,a_2,a_3,a_4,a_5,a_6,z_{\pm})} = 1,
\end{equation*}
c.f. the definition of $z_{\pm}$ and formula (\ref{eqZpm}).

Finally, we get
\begin{equation*}
\frac{\partial \mathscr{V}}{\partial \theta_1} = - \frac{a_1}{4} \frac{\partial \mathscr{U}}{\partial a_1}\bigg{\vert}^{z=z_{-}}_{z=z_{+}} 
= -\frac{1}{4} \log \left|\frac{\phi(z_{-}) \psi(z_{+})}{\phi(z_{+}) \psi(z_{-})}\right| + \frac{i \pi}{2} k,
\end{equation*}
for a certain $k \in \mathbb{Z}$, where the functions $\phi(\circ)$ and $\psi(\circ)$ are
\begin{equation*}
\phi(z) = (1+a_1a_2a_3z)(1+a_1a_5a_6z),
\end{equation*}
\begin{equation*}
\psi(z) = (1-a_1a_2a_4a_5z)(1-a_1a_3a_4a_6z).
\end{equation*}

The real part of the above expression is
\begin{equation}\label{eqdVdtheta1}
\Re \frac{\partial \mathscr{V}}{\partial \theta_1} = -\frac{1}{4} \log \left\vert\frac{\phi(z_{-}) \psi(z_{+})}{\phi(z_{+}) \psi(z_{-})}\right\vert.
\end{equation}

Let us set $\Delta = \det G$ and $\delta = \sqrt{\det G}$. We shall show that the expression
\begin{equation}\label{eqE}
\mathscr{E} = \phi(z_{-}) \psi(z_{+}) \left( c_{34} + \delta \frac{a_1 - 1/a_1}{2} \right) - \phi(z_{+}) \psi(z_{-}) \left( c_{34} - \delta \frac{a_1 - 1/a_1}{2} \right)
\end{equation}
is identically zero for all $a_k \in \mathbb{C}$, $k=\overline{1,6}$, which it actually depends on.

In order to perform the computation, the following formulas are used:
\begin{equation*}
z_{-} = \frac{-\hat{q}_1 - 4 \delta}{2 \hat{q}_2},\,\,\, z_{+} = \frac{-\hat{q}_1 + 4 \delta}{2 \hat{q}_2},
\end{equation*}
where $\hat{q}_l = q_l/\prod^{6}_{k=1}a_k$ for $l=1,2$ (c.f. formulas (\ref{eqZpm}) -- (\ref{eqQ1Q2Q3Definition})).

Note that one may consider the expression $\mathscr{E}$ as a rational function of independent variables $a_k$, $k=\overline{1,6}$, $\Delta$ and $\delta$, 
making the computation easier to perform by a software routine \cite{W}. First, we have 
\begin{equation*}
\frac{1}{\delta} \frac{\partial \mathscr{E}}{\partial a_1} = \frac{4 a_1 \mathscr{Y}}{\hat{q}^3_2} \frac{\partial \Delta}{\partial a_1},
\end{equation*}
where $\mathscr{Y} = \mathscr{Y}(a_1,a_2,a_3,a_4,a_5,a_6)$ is a certain technical term explained in Appendix.
Since $\frac{\partial \Delta}{\partial a_1} = 2 \delta \frac{\partial \delta}{\partial a_1}$, the above expression gives us
\begin{equation}\label{eqdEda1}
\frac{\partial \mathscr{E}}{\partial a_1} = \frac{8 a_1 \Delta}{\hat{q}^3_2} \frac{\partial \delta}{\partial a_1} \mathscr{Y}.
\end{equation}

Second, we obtain
\begin{equation}\label{eqdEddelta}
\frac{\partial \mathscr{E}}{\partial \delta} = -\frac{8 a_1 \Delta}{\hat{q}^3_2} \mathscr{Y}.
\end{equation}

Finally, we recall that $\delta$ is a function of $a_k$, $k=\overline{1,6}$, and the total derivative of $\mathscr{E}$ with respect to $a_1$ is
\begin{equation*}
\frac{\partial}{\partial a_1}\bigg( \mathscr{E}\vert_{\delta := \delta(a_1,a_2,a_3,a_4,a_5,a_6)} \bigg) 
= \frac{\partial \mathscr{E}}{\partial a_1} + \frac{\partial \mathscr{E}}{\partial \delta} \frac{\partial \delta}{\partial a_1} = 0,
\end{equation*}
according to equalities (\ref{eqdEda1})--(\ref{eqdEddelta}).

An analogous computation shows that $\frac{\partial}{\partial a_k}\bigg( \mathscr{E} \vert_{\delta := \delta(a_1,a_2,a_3,a_4,a_5,a_6)} \bigg) = 0$ 
for all $k=\overline{1,6}$. Then by setting $a_k=1$, $k=\overline{1,6}$, we get $\Delta = -16$ and so $\delta = 8i$. In this case $\mathscr{E} = 0$. 

Thus the equality $\mathscr{E}\equiv 0$ holds for all $a_k \in \mathbb{C}$, $k=\overline{1,6}$. Together with (\ref{eqdVdtheta1}) it gives
\begin{equation}\label{eqdVdtheta1-2}
\Re \frac{\partial \mathscr{V}}{\partial \theta_1} = -\frac{1}{4} \log \left| \frac{c_{34}-\delta \frac{a_1-a^{-1}_1}{2}}{c_{34}+\delta \frac{a_1-a^{-1}_1}{2}} \right|.
\end{equation}

On the other hand, by formula (\ref{eqGramMinors}), we have
\begin{equation*}
g^{\star}_{34} = -\cosh \ell_1 = \frac{-c_{34}}{\sqrt{c_{33}} \sqrt{c_{44}}}.
\end{equation*}
Here, both $c_{33}$ and $c_{44}$ are positive by Proposition~\ref{propositionVertices}, since the vertices $v_3$ and $v_4$ are proper. The formula above leads to the following equation
\begin{equation*}
e^{2 \ell_1} + 2\, g^{\star}_{34}\, e^{\ell_1} + 1 = 0,
\end{equation*}
the solution to which is determined by
\begin{equation}\label{eqEll1}
e^{2 \ell_1} = \frac{c_{34}+\delta \frac{a_1-a^{-1}_1}{2}}{c_{34}-\delta \frac{a_1-a^{-1}_1}{2}},
\end{equation}
in analogy to \cite[Equation~5.3]{Ushijima}. Thus, equalities (\ref{eqdVdtheta1-2}) -- (\ref{eqEll1}) imply 
\begin{equation*}
\Re \frac{\partial \mathscr{V}}{\partial \theta_1} = -\frac{1}{4} \log\,e^{-2\ell_1} = \frac{\ell_1}{2}.
\end{equation*}
Together with Lemma~\ref{lemmaWTerm}, this implies that claim (i) is satisfied.

\smallskip
As already mentioned, in case (ii) it suffices to prove $\frac{\partial}{\partial \theta_2} \mathrm{Vol}\,T = -\frac{\ell_2}{2}$. 
The statement for another $k\in\{3,5,6\}$ is analogous.

By formula (\ref{eqGramMinors}) we have that
\begin{equation*}
g^{\star}_{24} = i \sinh \ell_2 = \frac{-c_{24}}{\sqrt{c_{22}}\sqrt{c_{44}}}.
\end{equation*}
Since the vertex $v_2$ is ultra-ideal and the vertex $v_4$ is proper, by Proposition~\ref{propositionVertices}, $c_{22}<0$ and $c_{44}>0$. Thus,
\begin{equation*}
\sinh \ell_2 = \frac{c_{24}}{\sqrt{-c_{22}c_{44}}}.
\end{equation*}
The formula above implies 
\begin{equation}\label{eqEll2-1}
e^{2\ell_2} = -\frac{c^2_{24} + 2c_{24}\sqrt{c^2_{24}-c_{22}c_{44}}+c^2_{24}-c_{22}c_{44}}{c_{22}c_{44}}
\end{equation}
By applying Jacobi's theorem \cite[Th\'eor\`eme 2.5.2]{Prasolov} to the Gram matrix $G$, we have
\begin{equation}\label{eqJacobiEll2}
c^2_{24} - c_{22}c_{44} = \Delta \left( \frac{a_2 - a^{-1}_2}{2} \right)^2.
\end{equation}
Combining (\ref{eqEll2-1}) -- (\ref{eqJacobiEll2}) together, it follows that
\begin{equation}\label{eqEll2-2}
- e^{2 \ell_2} = \frac{c_{24} + \delta \frac{a_2 - a^{-1}_2}{2}}{c_{24} - \delta \frac{a_2 - a^{-1}_2}{2}}.
\end{equation}

By analogy with (\ref{eqdVdtheta1}), we get the formula
\begin{equation}\label{eqdVdtheta2}
\Re \frac{\partial \mathscr{V}}{\partial \theta_2} = - \frac{1}{4} \log \left| \frac{\phi(z_{-}) \psi(z_{+})}{\phi(z_{+}) \psi(z_{-})} \right|,
\end{equation}
where
\begin{equation*}
\phi(z) = (1 + a_1 a_2 a_3 z)(1 + a_2 a_4 a_6 z),
\end{equation*}
\begin{equation*}
\psi(z) = (1 - a_1 a_2 a_4 a_5 z)(1 - a_2 a_3 a_5 a_6 z).
\end{equation*}

Similar to case (i), the following relation holds:
\begin{equation*}
\frac{\phi(z_{-}) \psi(z_{+})}{\phi(z_{+}) \psi(z_{-})} = \frac{c_{24} - \delta \frac{a_2-a^{-1}_2}{2}}{c_{24} + \delta \frac{a_2-a^{-1}_2}{2}}.
\end{equation*}

Then formulas (\ref{eqEll2-2}) -- (\ref{eqdVdtheta2}) yield
\begin{equation*}
\Re \frac{\partial \mathscr{V}}{\partial \theta_2} = -\frac{1}{4} \log \left| \frac{-1}{e^{2\ell_2}} \right| = \frac{\ell_2}{2}.
\end{equation*}
The first equality of case (ii) now follows. 
Carrying out an analogous computation for $\Re \frac{\partial \mathscr{V}}{\partial \theta_k}$, $k=3,5$, we obtain
\begin{equation*}
\Re \frac{\partial \mathscr{V}}{\partial \theta_3} =  \frac{\ell_3}{2} \mbox{ and } \Re \frac{\partial \mathscr{V}}{\partial \theta_5} =  \frac{\ell_5}{2}.
\end{equation*}
Thus, all equalities of case (ii) hold.

As before, by formula (\ref{eqGramMinors}), we obtain that
\begin{equation}\label{eqGStar12}
- g^{\star}_{12} = \cos \mu = \frac{c_{12}}{\sqrt{c_{11}} \sqrt{c_{22}}}.
\end{equation}
Since both vertices $v_1$ and $v_2$ are ultra-ideal, by Proposition~\ref{propositionVertices}, the cofactors $c_{11}$ and $c_{22}$ are negative. Then (\ref{eqGStar12}) implies the equation
\begin{equation*}
e^{2 i \mu} + \frac{2 c_{12} e^{i \mu}}{\sqrt{c_{11}c_{22}}} + 1 = 0.
\end{equation*}
Without loss of generality, the solution we choose is 
\begin{equation}\label{eqMu}
e^{i \mu} = \frac{-c_{12}+\sqrt{c^2_{12}-c_{11}c_{22}}}{\sqrt{c_{11} c_{22}}}.
\end{equation}

By squaring (\ref{eqMu}) and by applying Jacobi's theorem to the corresponding cofactors of the matrix $G$, the formula below follows:
\begin{equation}\label{eqJacobiMu}
e^{2 i \mu} = \frac{c_{12} - \delta \frac{a_4 - a^{-1}_4}{2}}{c_{12} + \delta \frac{a_4 - a^{-1}_4}{2}}.
\end{equation}

On the other hand, a direct computation shows that
\begin{equation}\label{eqdVdl}
\Re \frac{\partial \mathscr{V}}{\partial \ell} = \frac{i}{4}\, \log\left( \frac{\phi(z_{-}) \psi(z_{+})}{\phi(z_{+}) \psi(z_{-})} \right),
\end{equation}
where
\begin{equation*}
\phi(z) = (1 + a_3 a_4 a_5 z)(1 + a_2 a_4 a_6 z),
\end{equation*}
\begin{equation*}
\psi(z) = (1 - a_1 a_2 a_4 a_5 z)(1 - a_1 a_3 a_4 a_6 z).
\end{equation*}

Similar to cases (i) and (ii), we have the relation
\begin{equation*}
\frac{\phi(z_{-}) \psi(z_{+})}{\phi(z_{+}) \psi(z_{-})} = \frac{c_{12} - \delta \frac{a_4 - a^{-1}_4}{2}}{c_{12} + \delta \frac{a_4 - a^{-1}_4}{2}},
\end{equation*}
which together with (\ref{eqJacobiMu}) -- (\ref{eqdVdl}) yields
\begin{equation*}
\frac{\partial \mathscr{V}}{\partial \ell} = \frac{i}{4} \log e^{2 i \mu} = -\frac{\mu}{2} - \frac{\pi}{2}\,k, \mbox{  } k\in \mathbb{Z}.
\end{equation*}
Since, $0 \leq \mu \leq \pi$, we choose $k=0$ and so $\frac{\partial \mathscr{V}}{\partial \ell} = -\frac{\mu}{2}$. The latter formula implies the equality of case (iii).

Now, by the generalised Schl\"{a}fli formula \cite{MilnorSchlaefli},
\begin{equation*}
\mathrm{d} \mathrm{Vol}\,T = -\frac{1}{2}\sum_{k\in\{1,2,3,5,6\}} \ell_k \mathrm{d}\theta_k - \frac{\ell}{2} \mathrm{d}\mu.
\end{equation*}

Together with the equalities of cases (i)-(iii) it yields
\begin{equation}\label{eqVolume}
\mathrm{Vol}\,T = \Re \left(-\mathscr{V} + \mathscr{W} - \frac{\mu \ell}{2} \right) + \mathscr{C},
\end{equation}
where $\mathscr{C} \in \mathbb{R}$ is a constant.

Finally, we prove that $\mathscr{C} = 0$, and the theorem follows. Passing to the limit $\theta_k \rightarrow \frac{\pi}{2}$, $k=\overline{1,6}$, 
the generalised hyperbolic tetrahedron $T$ shrinks to a point, since geometrically it tends to a Euclidean prism. 
Thus, we have $\mu \rightarrow \frac{\pi}{2}$ and $\ell \rightarrow 0$. By setting the limiting values above, 
we obtain that $a_k = i$, $k\in\{1,2,3,5,6\}$, $a_4 = 1$. Then $z_{-} = z_{+} = 1$ by equality (\ref{eqZpm}), 
and hence $\mathscr{V} = 0$. Since the dilogarithm function does not branch at $\pm 1$, $\pm i$, we have $\mathscr{W} = 0$.
Since $T$ shrinks to a point, $\mathrm{Vol}\,T \rightarrow 0$, that implies $\mathscr{C} = 0$ by means of (\ref{eqVolume}) and the proof is completed.~$\square$

\begin{table}[ht]
\begin{center}
\begin{tabular}{|c|c|c|c|}
\hline  $(\theta_1,\theta_2,\theta_3,\theta_5,\theta_6)$& $(\ell, \mu)$& Volume& Reference\\
\hline  $(\frac{\pi}{2}, \frac{\pi}{3}, \frac{\pi}{3}, \frac{\pi}{2}, \frac{\pi}{2})$&  $(0, 0)$&  $0.5019205$&  \cite{DK}\\
\hline  $(\frac{\pi}{2}, \frac{\pi}{3}, \frac{\pi}{3}, \frac{\pi}{2}, \frac{\pi}{2})$&  $(0.3164870, \frac{\pi}{4})$&  $0.4438311$&  \cite{DK}\\ 
\hline  $(\frac{\pi}{2}, \frac{\pi}{3}, \frac{\pi}{4}, \frac{\pi}{2}, \frac{\pi}{2})$&  $(0, 0)$&  $0.6477716$&  \cite{DK}\\ 
\hline  $(\frac{\pi}{2}, \frac{\pi}{3}, \frac{\pi}{4}, \frac{\pi}{2}, \frac{\pi}{2})$&  $(0.3664289, \frac{\pi}{4})$&  $0.5805842$&  \cite{DK}\\
\hline  $(\frac{\pi}{2}, \frac{\pi}{3}, \frac{\pi}{5}, \frac{\pi}{2}, \frac{\pi}{2})$&  $(0, 0)$&  $0.7466394$&  \cite{DK}\\ 
\hline  $(\frac{\pi}{2}, \frac{\pi}{3}, \frac{\pi}{5}, \frac{\pi}{2}, \frac{\pi}{2})$&  $(0.3835985, \frac{\pi}{4})$&  $0.6764612$&  \cite{DK}\\
\hline  $(\frac{\pi}{2}, 0, \frac{\pi}{2}, 0, \frac{\pi}{2})$&  $(0, 0)$&  $0.9159659$&  \cite{Sz}\\
\hline  $(\frac{\pi}{2}, \frac{\pi}{3}, \frac{\pi}{2}, \frac{\pi}{3}, \frac{\pi}{2})$&  $(0.5435350, \frac{\pi}{3})$&  $0.3244234$&  \cite{Sz}\\
\hline  $(\frac{\pi}{2}, \frac{\pi}{4}, \frac{\pi}{2}, \frac{\pi}{4}, \frac{\pi}{2})$&  $(0.5306375, \frac{\pi}{4})$&  $0.5382759^{\flat}$&  \cite{Sz}\\
\hline  $(\frac{\pi}{2}, \frac{\pi}{5}, \frac{\pi}{2}, \frac{\pi}{5}, \frac{\pi}{2})$&  $(0.4812118, \frac{\pi}{5})$&  $0.6580815$&  \cite{Sz}\\
\hline  $(\frac{\pi}{2}, \frac{\pi}{6}, \frac{\pi}{2}, \frac{\pi}{6}, \frac{\pi}{2})$&  $(0.4312773, \frac{\pi}{6})$&  $0.7299264$&  \cite{Sz}\\
\hline  $(\frac{\pi}{2}, \frac{\pi}{10}, \frac{\pi}{2}, \frac{\pi}{10}, \frac{\pi}{2})$&  $(0.2910082, \frac{\pi}{10})$&  $0.8447678$&  \cite{Sz}\\
\hline 
\end{tabular} 
\\
\smallskip
\begin{small}
${}^{\flat}$ this value is misprinted in \cite[Equation~4.11]{Sz}
\end{small}
\end{center}
\caption{Some numerically computed volume values} \label{tabularVolumes}
\end{table}

\section{Numeric computations}

In Table~\ref{tabularVolumes} we have collected several volumes of prism truncated tetrahedron computed by using Theorem~\ref{thmVolume} and before in \cite{DK, Sz}. The altitude of a prism truncated tetrahedron is computed from its dihedral angles, as the remark after Theorem~\ref{thmVolume} states. All numeric computations are carried out using the software routine ``Mathematica'' \cite{W}.

\section*{Appendix}

\medskip
\textbf{The term $\mathscr{Y}$ from Section~3.2}

Let us first recall of the expressions $q_k$, $k=0,1,2$, that are polynomials in the variables $a_k$, $k=\overline{1,6}$ defined by formula (\ref{eqQ1Q2Q3Definition}) 
and the expressions $\hat{q}_l$, $l=1,2$, defined in the proof of Theorem~\ref{thmVolume} by $\hat{q}_l = q_l/\prod^{6}_{k=1}a_k$, $l=1,2$. 
Then, the following lemma holds concerning the technical term $\mathscr{Y}$ mentioned above, that actually equals
\begin{eqnarray}\label{eqYTerm}
\nonumber \mathscr{Y} = a_1 a_2^2 a_3 a_4^2 a_5 + a_1^2 a_2 a_3^2 a_4^2 a_5 + a_2^3 a_3^2 a_4^2 a_5 + a_1 a_2^2 a_3^3 a_4^2 a_5 + a_1 a_2^2 a_4 a_5^2\\ 
\nonumber + a_1^2 a_2 a_3 a_4 a_5^2 + a_2^3 a_3 a_4 a_5^2 + a_1 a_2^2 a_3^2 a_4 a_5^2 + a_1^2 a_2^2 a_3 a_4^2 a_6 + a_1 a_2 a_3^2 a_4^2 a_6\\ 
\nonumber + a_1 a_2^3 a_3^2 a_4^2 a_6 + a_2^2 a_3^3 a_4^2 a_6 + a_1^2 a_2^2 a_4 a_5 a_6 + a_1 a_2 a_3 a_4 a_5 a_6 + a_1^3 a_2 a_3 a_4 a_5 a_6\\  
\nonumber + a_1 a_2^3 a_3 a_4 a_5 a_6 + a_1^2 a_3^2 a_4 a_5 a_6 + 2 a_2^2 a_3^2 a_4 a_5 a_6 - a_1^2 a_2^2 a_3^2 a_4 a_5 a_6 + a_1^4 a_2^2 a_3^2 a_4 a_5 a_6\\ 
\nonumber + a_1 a_2 a_3^3 a_4 a_5 a_6 + a_1 a_2 a_3 a_4^3 a_5 a_6 - a_1^3 a_2 a_3 a_4^3 a_5 a_6 + a_2^2 a_3^2 a_4^3 a_5 a_6 - a_1^2 a_2^2 a_3^2 a_4^3 a_5 a_6\\ 
\nonumber + a_2^2 a_3 a_5^2 a_6 - a_1^2 a_2^2 a_3 a_5^2 a_6 + a_1 a_2 a_3^2 a_5^2 a_6 - a_1^3 a_2 a_3^2 a_5^2 a_6 + a_1 a_2 a_4^2 a_5^2 a_6\\
\nonumber + a_1^2 a_3 a_4^2 a_5^2 a_6 + 2 a_2^2 a_3 a_4^2 a_5^2 a_6 - a_1^2 a_2^2 a_3 a_4^2 a_5^2 a_6 + a_1^4 a_2^2 a_3 a_4^2 a_5^2 a_6 + a_1 a_2 a_3^2 a_4^2 a_5^2 a_6\\
\nonumber + a_1^3 a_2 a_3^2 a_4^2 a_5^2 a_6 + a_1 a_2^3 a_3^2 a_4^2 a_5^2 a_6 + a_1^2 a_2^2 a_3^3 a_4^2 a_5^2 a_6 + a_2^2 a_4 a_5^3 a_6 + a_1 a_2 a_3 a_4 a_5^3 a_6\\
\nonumber + a_1 a_2^3 a_3 a_4 a_5^3 a_6 + a_1^2 a_2^2 a_3^2 a_4 a_5^3 a_6 + a_1^2 a_2 a_3 a_4 a_6^2 + a_1 a_3^2 a_4 a_6^2 + a_1 a_2^2 a_3^2 a_4 a_6^2\\
\nonumber + a_2 a_3^3 a_4 a_6^2 + a_1 a_2^2 a_3 a_5 a_6^2 - a_1^3 a_2^2 a_3 a_5 a_6^2 + a_2 a_3^2 a_5 a_6^2 - a_1^2 a_2 a_3^2 a_5 a_6^2\\
\nonumber + a_1^2 a_2 a_4^2 a_5 a_6^2 + a_1 a_3 a_4^2 a_5 a_6^2 + a_1 a_2^2 a_3 a_4^2 a_5 a_6^2 + a_1^3 a_2^2 a_3 a_4^2 a_5 a_6^2 + 2 a_2 a_3^2 a_4^2 a_5 a_6^2\\
\nonumber - a_1^2 a_2 a_3^2 a_4^2 a_5 a_6^2 + a_1^4 a_2 a_3^2 a_4^2 a_5 a_6^2 + a_1^2 a_2^3 a_3^2 a_4^2 a_5 a_6^2 + a_1 a_2^2 a_3^3 a_4^2 a_5 a_6^2 + a_1 a_2^2 a_4 a_5^2 a_6^2\\
\nonumber + 2 a_2 a_3 a_4 a_5^2 a_6^2 - a_1^2 a_2 a_3 a_4 a_5^2 a_6^2 + a_1^4 a_2 a_3 a_4 a_5^2 a_6^2 + a_1^2 a_2^3 a_3 a_4 a_5^2 a_6^2 + a_1 a_3^2 a_4 a_5^2 a_6^2\\
\nonumber + a_1 a_2^2 a_3^2 a_4 a_5^2 a_6^2 + a_1^3 a_2^2 a_3^2 a_4 a_5^2 a_6^2 + a_1^2 a_2 a_3^3 a_4 a_5^2 a_6^2 + a_2 a_3 a_4^3 a_5^2 a_6^2 - a_1^2 a_2 a_3 a_4^3 a_5^2 a_6^2\\
\nonumber + a_1 a_2^2 a_3^2 a_4^3 a_5^2 a_6^2 - a_1^3 a_2^2 a_3^2 a_4^3 a_5^2 a_6^2 + a_2 a_4^2 a_5^3 a_6^2 + a_1 a_3 a_4^2 a_5^3 a_6^2 + a_1 a_2^2 a_3 a_4^2 a_5^3 a_6^2\\
\nonumber + a_1^2 a_2 a_3^2 a_4^2 a_5^3 a_6^2 + a_1 a_2 a_3 a_4 a_5 a_6^3 + a_3^2 a_4 a_5 a_6^3 + a_1^2 a_2^2 a_3^2 a_4 a_5 a_6^3 + a_1 a_2 a_3^3 a_4 a_5 a_6^3\\
\nonumber + a_1 a_2 a_4^2 a_5^2 a_6^3 + a_3 a_4^2 a_5^2 a_6^3 + a_1^2 a_2^2 a_3 a_4^2 a_5^2 a_6^3 + a_1 a_2 a_3^2 a_4^2 a_5^2 a_6^3.
\end{eqnarray}

\begin{lemma}\label{lemmaTechnicalTerm}
The above expression $\mathscr{Y}$ is not a polynomial neither in the variables $q_0$, $q_1$, $q_2$, nor in the variables $q_0$, $q_1$, $\hat{q}_2$.
\end{lemma}
\begin{proof}
By setting $a_1 = a_6 := 0$, we get $q_0 = 1 + a_3 a_4 a_5$, $q_1 = q_2 = 0$ and
\begin{equation*}
\mathscr{Y}{|}_{a_1 = a_6 := 0} = a_2^3 a_3 a_4 a_5 (a_3 a_4 + a_5) =  a_2^3\, (q_0{|}_{a_1 = a_6 := 0} - 1)\, (a_3 a_4 + a_5).
\end{equation*}
As well, we have $\hat{q}_2 = a_2 (a_3 a_4 + a_5)$ and
\begin{equation*}
\mathscr{Y}{|}_{a_1 = a_6 := 0} = a_2^2\, (q_0{|}_{a_1 = a_6 := 0} - 1)\, \hat{q}_2{|}_{a_1 = a_6 := 0}.
\end{equation*}
The former equality proves that $\mathscr{Y}$ is not a polynomial in the variables $q_0$, $q_1$, $q_2$. The latter shows that $\mathscr{Y}$ is not a polynomial in $q_0$, $q_1$, $\hat{q}_2$ either.
\end{proof}

\bigskip

{\it
Alexander Kolpakov

Department of Mathematics

University of Fribourg

chemin du Mus\'ee 23

CH-1700 Fribourg, Switzerland

kolpakov.alexander(at)gmail.com}

\medskip
{\it
Jun Murakami

Department of Mathematics

Faculty of Science and Engineering

Waseda University

3-4-1 Okubo Shinjuku-ku 

169-8555 Tokyo, Japan

murakami(at)waseda.jp}

\end{document}